\documentclass[12pts]{article}%
\usepackage{amsfonts}
\usepackage{amsmath}
\usepackage{amssymb}
\usepackage{graphicx}%
\providecommand{\U}[1]{\protect\rule{.1in}{.1in}}
\newtheorem{theorem}{Theorem}

\newtheorem{definition}[theorem]{Definition}

\newtheorem{remark}[theorem]{Remark}

\newenvironment{proof}[1][Proof]{\noindent\textbf{#1.} }{\ \rule{0.5em}{0.5em}}
\begin{document}

\begin{center}
{\Large A common framework for some techniques in }

{\Large Applied mathematics}

{\LARGE \ } \\[0pt]\vspace{1cm} {\large by\vspace{1cm}}

{\large \textbf{Bilal Chanane}}

\vspace{1cm}

Department of Mathematics and Statistics

KFUPM, Dhahran 31261, Saudi Arabia\\[0pt]E-Mail: chanane@kfupm.edu.sa\\[0pt]%
\vspace{1cm}
\end{center}

\textbf{Abstract}

The objective in this paper is to demonstrate that four of the most used
techniques in applied mathematics, viz., Fourier series, Fourier transform,
Laplace transform and the Fourier-Laplace transform can be introduced using
eigenvalue problems for first order differential operators with
discrete/continuous spectra.

\textbf{Key Words }Discrete spectrum, continuous spectrum, Fourier series,
Fourier transform, Laplace transform, Fourier-Laplace transform, eigenvalue
problems, eigenvalue problems, eigenfunctions expansion.

\textbf{AMS subject classification} \ 42A16, 42B05, 42A38, 44A10, 34L10

\section{Introduction}

\setcounter{equation}{0}

A look at several applied/engineering mathematics textbooks (e.g.,
\cite{F69},\cite{TB94}, \cite{O1995}, \cite{S96}, \cite{ZC2000}, \cite{D2003}
) reveals that four of the most important techniques in applied mathematics,
viz., Fourier series, Fourier transform, Laplace transform and Fourier-Laplace
transform are introduced somewhat independently though some kind of limiting
argument is used to go from Fourier series to Fourier transform. More often
the Fourier series is introduced as arising from a second order
Sturm-Liouville problem. We note however, in some instances, its introduction
through first order eigenvalue problem is mentioned. It seems that no textbook
has approached all four problems using the same method. We shall present in
the following sections these notions as arising from from first order
eigenvalue problems with discrete/continuous spectra. This is by no means a
substitute for a more sophisticated presentation involving higher mathematics
to which students are not introduced until very late in their curricula, if at
all. Having said that, we think that the presentation using this approach
allows one to deal in a unified manner and at an elementary level these
fundamental tools and it should be adequate for engineering and applied
mathematics students.

\section{The Fourier series}

\setcounter{equation}{0}

Consider the problem of finding the values of the parameter $\lambda$ for
which the following problem
\begin{equation}
\left\{
\begin{array}
[c]{c}%
i\frac{dy}{dx}=\lambda y\mbox{ , }x\in(-L,L)\\
y(-L)=y(L)
\end{array}
\right.  \label{eq1}%
\end{equation}

\bigskip will have non trivial solutions, that is solutions that are not
identically zero. Such a parameter $\lambda$ is called an eigenvalue, while
the corresponding non zero solution is called an eigenfunction belonging to
the eigenvalue $\lambda$.

\bigskip

\begin{theorem}
The eigenvalue problem (\ref{eq1}) has an inifinite sequence of eigenvalues
$\lambda_{k}=k\pi/L$, $k\in\mathbb{Z}$, where $\mathbb{Z}$ is the set of
relative integers. The corresponding eigenfunctions $y_{k}=\exp(-ik\pi x/L)$,
$k\in\mathbb{Z,}$ are orthogonal with respect to the inner product
$<f,g>=\int_{-L}^{L}f(x)\overline{g(x)}dx$. The set of eigenfunctions is
complete and any function $f$ satisfying $\int_{-L}^{L}\left\vert
f(x)\right\vert ^{2}dx<\infty$, will have the expansion
\[
f(x)\sim\sum_{k=-\infty}^{\infty}c_{k}\mathbf{e}^{-i\frac{k\pi}{L}x}
\]
where%
\[
c_{k}=\frac{1}{2L}\int_{-L}^{L}f(x)\mathbf{e}^{i\frac{k\pi}{L}x}dx\text{,
}k\in\mathbb{Z}\text{.}
\]
The above series is called the complex Fourier series of $f$.
\end{theorem}

\begin{proof}
From the differential equation we have $dy/dx=-i\lambda y$ whose general
solution is $y=c\exp(-i\lambda x)$. Using the boundary condition, we get
$c\exp(i\lambda L)=c\exp(-i\lambda L)$ leading to $2ic\sin\left(  \lambda
L\right)  =0$. To have non trivial solutions a necessary and sufficient
condition is $\sin\left(  \lambda L\right)  =0$. therefore, $\lambda_{k}%
=\frac{k\pi}{L},k\in\mathbb{Z}$ are the eigenvalues of the problem. The
corresponding eigenfunctions are $y_{k}=\exp(-i\frac{k\pi x}{L}),k\in
\mathbb{Z}.$Note that we have taken the constant $c=1$ as any non zero
multiple of an eigenfunction is an eigenfunction itself. The eigenfunctions
$y_{k}$ $,k\in\mathbb{Z},$ are orthogonal with respect to the inner product%
\begin{equation}
<f,g>=\int_{-L}^{L}f(x)\overline{g(x)}dx \label{eq4}%
\end{equation}
Indeed,
\begin{align}
&  <\exp(-i\frac{k\pi x}{L}),\exp(-i\frac{l\pi x}{L})>=\int_{-L}^{L}%
\exp(-i\frac{k\pi x}{L})\exp(i\frac{l\pi x}{L})dx\nonumber\\
&  =\int_{-L}^{L}\exp(-i\frac{\pi x}{L}\left\{  k-l\right\}  )dx\nonumber\\
&  =\left\{
\begin{array}
[c]{c}%
2L\mbox{  if }k=l\\
\frac{2L}{\pi\left\{  k-l\right\}  }\sin\left(  \pi\left\{  k-l\right\}
\right)  \mbox{ if }k\neq l
\end{array}
\right. \nonumber\\
&  =\left\{
\begin{array}
[c]{c}%
2L\mbox{ if }k=l\\
0\mbox{ if }k\neq l
\end{array}
\right.  \label{eq5}%
\end{align}
The set of eigenfunctions is complete since for if $f$ $\ $\ is any function
satisfying $\int_{-L}^{L}\left\vert f(x)\right\vert ^{2}dx<\infty$ then
$\int_{-L}^{L}f(x)\exp(i\frac{k\pi x}{L})dx=0$ for all $k\in Z$ will imply
$f(x)=0$ for almost all $x\in\lbrack-L,L]$.

Now, let $f$ be such that $\int_{-L}^{L}\left\vert f(x)\right\vert
^{2}dx<\infty$ , then it has the eigenfunctions expansion,%
\[
f(x)\sim\sum\limits_{k=-\infty}^{\infty}c_{k}\exp(-i\frac{k\pi x}{L})
\]
Multiplying by $\exp(i\frac{l\pi x}{L})$ and integrating both sides with
respect to $x$ from $-L$ to $L$, we get%
\[
\int_{-L}^{L}f(x)\exp(i\frac{l\pi x}{L})dx=2Lc_{l}
\]
that is,
\begin{equation}
c_{k}=\frac{1}{2L}\int_{-L}^{L}f(x)\exp(i\frac{k\pi x}{L})dx\mbox{ , }k\in
\mathbb{Z} \label{eq7}%
\end{equation}
which ends the proof.
\end{proof}

\bigskip

A connection with second order Sturm-Liouville problems goes like this.
Applying $id/dx$ to both sides of (\ref{eq1}), we get $-\frac{d^{2}y}{dx^{2}%
}=\lambda i\frac{dy}{dx}=\lambda^{2}y$. Since $y(-L)=y(L)$ we have%
\[
i\frac{dy}{dx}(-L)=\lambda y(-L)=\lambda y(L)=i\frac{dy}{dx}(L)\mbox{ , }
\]
that is%
\[
\frac{dy}{dx}(-L)=\frac{dy}{dx}(L)\mbox{.}
\]
Thus,%
\begin{equation}
\left\{
\begin{array}
[c]{c}%
-\frac{d^{2}y}{dx^{2}}=\lambda^{2}y\mbox{ , }x\in(-L,L)\\
y(-L)=y(L)\\
\frac{dy}{dx}(-L)=\frac{dy}{dx}(L)
\end{array}
\right.  \label{eq8}%
\end{equation}
a second order Sturm-Liouville problem which is usually taken as a point of
departure for the introduction of Fourier series. Its eigenvalues $\lambda
_{k}^{2}$ are just the square of the eigenvalues $\lambda_{k}$ of (\ref{eq1}),
$\lambda_{k}^{2}=\left(  \frac{k\pi}{L}\right)  ^{2}\mbox{ , }k\geq0$ and the
corresponding eigenfunctions are%
\[
1\mbox{ , }\left\{  \cos\frac{k\pi x}{L}\mbox{ , }\sin\frac{k\pi x}%
{L}\right\}  _{k\geq1}
\]
Noting that, $c_{-k}=\overline{c_{k}}\mbox{ , }k\geq1$\bigskip, one obtains
the real Fourier series as,
\begin{equation}
f(x)\sim\frac{a_{0}}{2}+\sum\limits_{k=1}^{\infty}\left\{  a_{k}\cos\frac{k\pi
x}{L}+b_{k}\sin\frac{k\pi x}{L}\right\}  \label{eq12}%
\end{equation}
where,
\begin{equation}
a_{k}=\frac{1}{L}\int_{-L}^{L}f(x)\cos(\frac{k\pi x}{L})dx\text{ , }%
b_{k}=\frac{1}{L}\int_{-L}^{L}f(x)\sin(\frac{k\pi x}{L})dx\text{.}
\label{eq11}%
\end{equation}

\section{The Fourier transform}

\setcounter{equation}{0}

\bigskip Consider the eigenvalue problem,
\begin{equation}
i\frac{dy}{dx}=\lambda y\mbox{ , }x\in\mathbb{R} \label{eq14}%
\end{equation}

From the differential equation we get $dy/dx=-i\lambda y$ whose solution is
$y=c\exp\left(  -i\lambda x\right)  $. To have a non trivial solution we need
$c\neq0$. We may as well take $c=1$ since any other non zero solution is just
a multiple of this one. Thus,
\begin{equation}
y_{\lambda}(x)=\exp\left(  -i\lambda x\right)  \label{eq15}%
\end{equation}
We notice that this function is not square integrable because no matter what
the value of $\lambda$ is,
\[
\int_{R}\left\vert y_{\lambda}(x)\right\vert ^{2}dx=\int_{R}\left\vert
\exp\left(  -i\lambda x\right)  \right\vert ^{2}dx=\int_{R}\exp\left(
2x\operatorname{Im}\lambda\right)  dx=\infty
\]
therefore, the operator $i\frac{d}{dx}$ has no eigenvalue as such. However, we
shall introduce the concepts of continuum eigenvalue and continuum
eigenfunction \cite{F69}.

\begin{definition}
\bigskip A number $\lambda$ is said to be a continuum eigenvalue for an
operator $M$ if there exists a sequence of functions $y_{n}$ in the domain of
$M$ such that the ratio $\frac{\left\vert \left\vert (M-\lambda)y_{n}%
\right\vert \right\vert }{\left\vert \left\vert y_{n}\right\vert \right\vert
}$ converges to zero as $n$ goes to $\infty$. If the functions $y_{n}$
converge pointwise to a function $y$, then $y$ is called a continuum
eigenfunction of $M$ corresponding to $\lambda$. We say then that $\lambda$
belongs to the continuous spectrum of $M$.
\end{definition}

\begin{remark}
If the convergence of $y_{n}$ is in the sense of the space to $y$, $\lambda$
would be an eigenvalue and $y$ a corresponding eigenfunction. In that case we
say that $\lambda$ belongs to the discrete spectrum of $M$.
\end{remark}

Returning to the eigenvalue problem \ref{eq14}), we claim,

\begin{theorem}
\bigskip The eigenvalue problem (\ref{eq14}) has a continuous spectrum given
by $\mathbb{R}$. Corresponding to the continuum eigenvalue $\lambda
\in\mathbb{R}$, we associate the continuum eigenfunction $y_{\lambda}%
(x)=\exp\left(  -i\lambda x\right)  $. The continuum eigenfunctions are
orthogonal with respect to the inner product $<f,g>=\int_{-\infty}^{\infty
}f(x)\overline{g(x)}dx$. The set of continuum eigenfunctions is complete and
any function $f$ satisfying $\int_{-\infty}^{\infty}\left\vert f(x)\right\vert
^{2}dx<\infty$, will have the representation
\[
f(x)\sim\int_{\mathbb{R}}F(\lambda)\exp\left(  -i\lambda x\right)  d\lambda
\]
where%
\[
F(\lambda)\sim\frac{1}{2\pi}\int_{-\infty}^{\infty}f(x)\exp\left(  i\lambda
x\right)  dx\text{.}
\]
$F$ is called the Fourier transform of $f$ and $f$ the inverse Fourier
transform of $F.$
\end{theorem}

\begin{proof}
We have here a continuous spectrum given by $\mathbb{R}$. As for the
orthogonality we have,%
\begin{align*}
\int_{R}y_{\lambda}(x)\overline{y_{\mu}(x)}dx  &  =\int_{R}\exp\left(
-i\lambda x\right)  \overline{\exp\left(  -i\mu x\right)  }dx\\
&  =\int_{R}\exp\left(  -i\lambda x\right)  \exp\left(  i\mu x\right)  dx\\
&  =\int_{R}\exp\left(  -i\left(  \lambda-\mu\right)  x\right)  dx\\
&  =2\pi\delta\left(  \lambda-\mu\right)
\end{align*}

\end{proof}

where we have made use of the property of the Dirac delta (generalized)
function $\delta$,
\begin{equation}
\delta(a)=\frac{1}{2\pi}\int_{R}\exp\left(  -iax\right)  dx \label{eq16}%
\end{equation}

Thus, for $f$ such that $\int_{R}\left\vert f(x)\right\vert ^{2}dx<\infty$ we
have the representation,%
\begin{equation}
f(x)=\int_{R}F(\lambda)\exp\left(  -i\lambda x\right)  d\lambda\label{eq18}%
\end{equation}

from which we get, after multiplication by $\exp\left(  i\mu x\right)  $ and
integration with respect to $x$ from $-\infty$ to $+\infty$,
\begin{align*}
\int_{R}f(x)\exp\left(  i\mu x\right)  dx  &  =\int_{R}\left\{  \int
_{R}F(\lambda)\exp\left(  -i\lambda x\right)  d\lambda\right\}  \exp\left(
i\mu x\right)  dx\\
&  =\int_{R}F(\lambda)\left\{  \int_{R}\exp\left(  -i\left(  \lambda
-\mu\right)  x\right)  dx\right\}  d\lambda\\
&  =\int_{R}F(\lambda)\left\{  2\pi\delta\left(  \lambda-\mu\right)  \right\}
d\lambda\\
&  =2\pi F(\mu)
\end{align*}

Here we used the property
\begin{equation}
\int_{R}g(z)\delta(z)dz=g(0) \label{eq19}%
\end{equation}

Thus,%
\[
F(\mu)=\frac{1}{2\pi}\int_{-\infty}^{\infty}f(x)\exp\left(  i\mu x\right)  dx
\]

which concludes the proof.

\section{\bigskip The Laplace transform}

\setcounter{equation}{0}

Consider the problem of finding the values of the parameter $\lambda$ for
which the following problem will have non trivial solutions,%

\begin{equation}
i\left(  \frac{dy}{dx}-\sigma y\right)  =\lambda y\mbox{ , }x\in
\lbrack0,\infty) \label{eq21}%
\end{equation}

We have $\frac{dy}{dx}=(\sigma-i\lambda)y$ thus, a non trivial solution is%

\begin{equation}
y_{\lambda}=\exp((\sigma-i\lambda)x) \label{eq22}%
\end{equation}

for any $\lambda\in R$ and any other non trivial solution is a multiple of
this one.

\bigskip Here again, we notice that this function is not square integrable
because no matter what the value of $\lambda$ is,
\[
\int_{R}\left\vert y_{\lambda}(x)\right\vert ^{2}dx=\int_{R}\left\vert
\exp((\sigma-i\lambda)x)\right\vert ^{2}dx=\int_{R}\exp\left(  2x\left\{
\sigma+\operatorname{Im}\lambda\right\}  \right)  dx=\infty
\]
therefore, the operator$i\left(  \frac{dy}{dx}-\sigma y\right)  $ has no
eigenvalue as such. We claim,

\begin{theorem}
The eigenvalue problem (\ref{eq21}) has a continuous spectrum given by
$\mathbb{R}$. Corresponding to the continuum eigenvalue $\lambda\in\mathbb{R}%
$, we associate the continuum eigenfunction $y_{\lambda}=\exp((\sigma
-i\lambda)x)$. The continuum eigenfunctions are orthogonal with respect to the
inner product $<f,g>=\int_{-\infty}^{\infty}\mathbf{e}^{-2\sigma
x}f(x)\overline{g(x)}dx$. The set of continuum eigenfunctions is complete and
any function $f$ satisfying $\int_{-\infty}^{\infty}\mathbf{e}^{-2\sigma
x}\left\vert f(x)\right\vert ^{2}dx<\infty$, will have the representation
\[
f(x)\sim\frac{1}{2\pi}\int_{-\infty}^{\infty}F(\lambda)e^{(\sigma-i\lambda
)x}d\lambda
\]
where%
\[
F(\lambda)\sim\int_{0}^{\infty}f(x)e^{(-\sigma+i\lambda)x}dx\text{.}
\]

\end{theorem}

\begin{proof}
We have here again a continuous spectrum given by $\mathbb{R}$. Any two
continuum eigenfunctions $y_{\lambda}$ and $y_{\mu}$ are orthogonal with
respect to the weight $w(x)=\exp(-2\sigma x)$ over $[0,\infty)$. Indeed,%
\begin{align}
\int_{0}^{\infty}e^{-2\sigma x}y_{\lambda}(x)\overline{y_{\mu}(x)}dx  &
=\int_{0}^{\infty}e^{-2\sigma x}e^{(\sigma-i\lambda)x}e^{(\sigma+i\mu
)x}dx\nonumber\\
&  =\int_{0}^{\infty}e^{-i(\lambda-\mu)x}dx\nonumber\\
&  =2\pi\delta(\lambda-\mu). \label{eq24}%
\end{align}
Let $f$ be such that%
\begin{equation}
\int_{0}^{\infty}\left\vert f(x)\right\vert ^{2}e^{-2\sigma x}dx<\infty.
\label{eq25}%
\end{equation}
We have,
\begin{equation}
f(x)=\frac{1}{2\pi}\int_{-\infty}^{\infty}F(\lambda)e^{(\sigma-i\lambda
)x}d\lambda. \label{eq26}%
\end{equation}
We have taken the factor $\frac{1}{2\pi}$ for convenience and compatibility
with known results. Multiplying by $\exp(-2\sigma x)\exp((\sigma+i\mu)x)$ and
integrating with respect to $x$ from $0$ to $\infty$, we get%
\begin{align}
\int_{0}^{\infty}f(x)\exp((-\sigma+i\mu)x)dx  &  =\int_{0}^{\infty}%
\exp((-\sigma+i\mu)x)\frac{1}{2\pi}\int_{-\infty}^{\infty}F(\lambda
)\exp((\sigma-i\lambda)x)d\lambda dx\nonumber\\
&  =\frac{1}{2\pi}\int_{-\infty}^{\infty}F(\lambda)2\pi\delta(\lambda
-\mu)d\lambda\nonumber\\
&  =F(\mu) \label{eq27}%
\end{align}
that is,%
\begin{equation}
F(\lambda)=\int_{0}^{\infty}f(x)e^{(-\sigma+i\lambda)x}dx\mbox{.} \label{eq28}%
\end{equation}

\end{proof}

\bigskip If we let $s=\sigma-i\lambda$ and denote $\widehat{f}(s)=F(\lambda)$,
we get,
\begin{equation}
\widehat{f}(s)=\int_{0}^{\infty}f(x)e^{-sx}dx\text{.} \label{eq29}%
\end{equation}

Now, $ds=-id\lambda$ so that
\begin{equation}
f(x)=\frac{1}{2\pi}\int_{\sigma+i\infty}^{\sigma-i\infty}\widehat{f}%
(s)e^{sx}ids \label{eq30}%
\end{equation}

leading to,%
\begin{equation}
f(x)=\frac{1}{2\pi i}\int_{\sigma-i\infty}^{\sigma+i\infty}\widehat
{f}(s)e^{sx}ds \label{eq31}%
\end{equation}

Form the above development we can see that for $\widehat{f}$ to exist, $f$ has
to satisfy%
\begin{equation}
\int_{0}^{\infty}\left\vert f(x)\right\vert ^{2}e^{-2\sigma x}dx<\infty
\label{eq32}%
\end{equation}

Thus there exists $M>0$ and $c>0$ such that $\left\vert f(x)\right\vert
^{2}e^{-2\sigma x}<M^{2}$ for all $x>c$. That is,%
\begin{equation}
\left\vert f(x)\right\vert <Me^{\sigma x}\mbox{ , for all }x>c \label{eq33}%
\end{equation}

We say that $f$ is of exponential type and $\sigma$ is called the abscissa of
convergence. Hence, we introduce,

\begin{definition}
Let $f$ be of exponential type, then the function $\widehat{f}$ $\ $defined
by,
\[
\widehat{f}(s)=\int_{0}^{\infty}f(x)e^{-sx}dx
\]
is called the Laplace transform of $f$. Furthermore,
\[
f(x)=\frac{1}{2\pi i}\int_{\sigma-i\infty}^{\sigma+i\infty}\widehat
{f}(s)e^{sx}ds
\]
gives the inverse Laplace transform of $\widehat{f}$.
\end{definition}

\section{\bigskip The Fourier-Laplace transform}

\setcounter{equation}{0}

Consider the problem of finding the values of the parameters $\mu$ and
$\lambda$ for which the following problem will have non trivial solutions,%

\begin{equation}
\left\{
\begin{array}
[c]{c}%
i\frac{\partial y}{\partial x}=\lambda y\\
i\left(  \frac{\partial y}{\partial t}-\sigma y\right)  =\mu y
\end{array}
\right.  \label{eq34}%
\end{equation}

$(x,t)\in(-\infty,\infty)\times(0,\infty)$. We claim,

\begin{theorem}
\bigskip The eigenvalue problem (\ref{eq34}) has a continuous spectrum given
by $\mathbb{R}^{2}$. Corresponding to the continuum eigenvalue $(\lambda
,\mu)\in\mathbb{R}^{2}$, we associate the continuum eigenfunction
$y_{\lambda,\mu}(x,t)=\exp(-i\lambda x+(\sigma-i\mu)t)$. These continuum
eigenfunctions are orthogonal with respect to the inner product $<f,g>=\int
_{-\infty}^{\infty}\int_{0}^{\infty}\mathbf{e}^{-2\sigma t}f(x)\overline
{g(t)}dxdt$. The set of continuum eigenfunctions is complete and any function
$f$ satisfying $\int_{-\infty}^{\infty}\int_{0}^{\infty}\mathbf{e}^{-2\sigma
t}\left\vert f(x,t)\right\vert ^{2}dxdt<\infty$, will have the representation
\[
f(x,t)\sim\frac{1}{2\pi}\frac{1}{2\pi i}\int_{\sigma-i\infty}^{\sigma+i\infty
}\int_{-\infty}^{\infty}F(\lambda,\mu)e^{-i\lambda x+\mu t}d\lambda d\mu
\]
where%
\[
F(\lambda,\mu)\sim\frac{1}{2\pi}\int_{-\infty}^{\infty}\int_{0}^{\infty
}f(x,t)e^{i\lambda x-\mu t}dxdt\text{.}
\]
$F$ is called the Fourier-Laplace transform of $f$ and $f$ is called the
inverse Fourier-Laplace transform of $F$.
\end{theorem}

\begin{proof}
The first differential equation in (\ref{eq34}) gives $y(x,t)=\exp(-i\lambda
x)u(t).$ Replacing into the second differential equation gives, $u^{\prime
}(t)=(\sigma-i\mu)u(t)$, that is $u(t)=c\exp((\sigma-i\mu)t).$Thus,%
\[
y_{\lambda,\mu}(x,t)=\exp(-i\lambda x+(\sigma-i\mu)t)
\]
where we have taken without loss of generality, $c=1$, is a continuum
eigenfunction corresponding to the continuum eigenvalue $(\lambda,\mu)$ in the
continuous spectrum $\mathbb{R}^{2}$. Any two continuum eigenfunctions
$y_{\lambda,\mu}(x,t)$ and $y_{\lambda^{\prime},\mu^{\prime}}(x,t)$ are
orthogonal with respect to the inner product $<f,g>=\int_{-\infty}^{\infty
}\int_{0}^{\infty}\mathbf{e}^{-2\sigma t}f(x)\overline{g(t)}dxdt$. Indeed,%
\begin{align*}
&  <y_{\lambda,\mu},y_{\lambda^{\prime},\mu^{\prime}}>=\int_{-\infty}^{\infty
}\int_{0}^{\infty}\mathbf{e}^{-2\sigma t}y_{\lambda,\mu}(x,t)\overline
{y_{\lambda^{\prime},\mu^{\prime}}(x,t)}dxdt\\
&  =\int_{-\infty}^{\infty}\int_{0}^{\infty}\mathbf{e}^{-2\sigma t}%
\exp(-i\lambda x+(\sigma-i\mu)t)\exp(i\lambda^{\prime}x+(\sigma+i\mu^{\prime
})t)dxdt\\
&  =\int_{-\infty}^{\infty}\int_{0}^{\infty}\exp(-i(\lambda-\lambda^{\prime
})x-i(\mu-\mu^{\prime})t)dxdt\\
&  =(2\pi)^{2}\delta(\lambda-\lambda^{\prime})\delta(\mu-\mu^{\prime})\text{.}%
\end{align*}
If $f$ is such that $\int_{-\infty}^{\infty}\int_{0}^{\infty}\mathbf{e}%
^{-2\sigma t}\left\vert f(x,t)\right\vert ^{2}dxdt<\infty$ then,%
\[
f(x,t)\sim\int_{-\infty}^{\infty}\int_{-\infty}^{\infty}\widetilde{F}%
(\lambda,\mu)e^{-i\lambda x+(\sigma-i\mu)t}d\lambda d\mu
\]
Multiplying both sides by $\mathbf{e}^{-2\sigma t}\exp(i\lambda^{\prime
}x+(\sigma+i\mu^{\prime})t)$ we get after integration with respect to $(x,t)$
over $\mathbb{R}\times\mathbb{R}_{+}$,
\[
\widetilde{F}(\lambda,\mu)=\frac{1}{(2\pi)^{2}}\int_{-\infty}^{\infty}\int
_{0}^{\infty}f(x,t)e^{i\lambda x+(\sigma+i\mu)t}dxdt
\]
Let $s=-(\sigma+i\mu)$ and $F(\lambda,s)=\widetilde{F}(\lambda,\mu)$, we
have,
\[
F(\lambda,s)=\frac{1}{(2\pi)^{2}}\int_{-\infty}^{\infty}\int_{0}^{\infty
}f(x,t)e^{i\lambda x-st}dxdt
\]
and%
\[
f(x,t)\sim\frac{1}{2\pi}\frac{1}{2\pi i}\int_{\sigma-i\infty}^{\sigma+i\infty
}\int_{-\infty}^{\infty}F(\lambda,\mu)e^{-i\lambda x+\mu t}d\lambda d\mu
\]
which ends the proof.
\end{proof}

\section{Conclusion}

In this paper we have provided a common framework to deal with four of the
most used techniques in applied mathematics, viz., Fourier series, Fourier
transform, Laplace transform and Fourier-Laplace transform. It has been shown
that they arise from first order eigenvalue problems with discrete/continuous
spectra. We believe that the approach is worth presenting in an introductory
course on applied/engineering mathematics.

\end{document}